\documentclass[a4paper,10pt,oneside]{amsart}

\usepackage{amsmath,amsfonts,amssymb}
\usepackage[all]{xy}
\usepackage[english]{babel}
\usepackage[utf8]{inputenc}
\usepackage{graphicx}
\usepackage{multirow}
\usepackage{array}
\usepackage{booktabs}
\usepackage{ifthen}
\usepackage{microtype}
\usepackage{lipsum}
\usepackage[hypertexnames=false,
backref=page,
    pdftex,
    pdfpagemode=UseNone,
    breaklinks=true,
    extension=pdf,
    colorlinks=true,
    linkcolor=blue,
    citecolor=blue,
    urlcolor=blue,
]{hyperref}

\reversemarginpar

\newcommand{\referenza}{}

\newtheorem{prop}{Proposition}[section]
\newtheorem{thm}[prop]{Theorem}
\newtheorem*{thm*}{Theorem \referenza}

\newtheorem{lemma}[prop]{Lemma}

\theoremstyle{definition}
\newtheorem{defi}[prop]{Definition}

\newcommand{\N}{\mathbb{N}}
\newcommand{\Z}{\mathbb{Z}}

\newcommand{\R}{\mathbb{R}}
\newcommand{\C}{\mathbb{C}}

\newcommand{\st}{\;:\;}

\newcommand{\sspace}{\cdot}
\newcommand{\ssspace}{\cdot\cdot}

\newcommand{\cp}{\smile}

\DeclareMathOperator{\im}{i}
\DeclareMathOperator{\imm}{im}
\DeclareMathOperator{\de}{d}

\DeclareMathOperator{\GL}{GL}

\newcommand{\del}{\partial}
\newcommand{\delbar}{\overline{\del}}

\newcommand{\CP}{\mathbb{CP}}

\DeclareMathOperator{\esp}{e}

\renewcommand{\Im}{\mathsf{Im}}
\renewcommand{\Re}{\mathsf{Re}}

\title{On Bott-Chern cohomology and formality}

\author{Daniele Angella}

\address[Daniele Angella]{Istituto Nazionale di Alta Matematica\\
at Dipartimento di Matematica e Informatica\\
Universit\`{a} di Parma\\
Parco Area delle Scienze 53/A, 43124\\
Parma, Italy}

\curraddr{Centro di Ricerca Matematica ``Ennio de Giorgi''\\
Collegio Puteano, Scuola Normale Superiore\\
Piazza dei Cavalieri 3\\
56126 Pisa, Italy}

\email{daniele.angella@gmail.com}

\author{Adriano Tomassini}

\address[Adriano Tomassini]{Dipartimento di Matematica e Informatica\\
Universit\`{a} di Parma \\
Parco Area delle Scienze 53/A, 43124 \\
Parma, Italy}

\email{adriano.tomassini@unipr.it}

\thanks{The first author is granted with a research fellowship by Istituto Nazionale di Alta Matematica INdAM, and is supported by the Project PRIN ``Varietà reali e complesse: geometria, topologia e analisi armonica'', by the Project FIRB ``Geometria Differenziale e Teoria Geometrica delle Funzioni'', and by GNSAGA of INdAM. The second author is supported by the Project PRIN ``Varietà reali e complesse: geometria, topologia e analisi armonica'', by Project FIRB ``Geometria Differenziale Complessa e Dinamica Olomorfa'', and by GNSAGA of INdAM.\vspace{10pt}\\
To appear in {\em J. Geom. Phys.}}

\begin{document}

\begin{abstract}
 We study a geometric notion related to formality for Bott-Chern cohomology on complex manifolds.
\end{abstract}

\maketitle

\section*{Introduction}

Cohomological properties of complex manifolds $X$ are encoded in several cohomological invariants. On the one side, the de Rham cohomology $H^{\bullet}_{dR}(X;\C)$ is a topological invariant. On the other side, the Dolbeault cohomology $H^{\bullet,\bullet}_{\delbar}(X)$ is not directly naturally linked with the de Rham cohomology. In a sense, {\em Bott-Chern and Aeppli cohomologies},
$$ H^{\bullet,\bullet}_{BC}(X) \;:=\; \frac{\ker\del\cap\ker\delbar}{\imm\del\delbar} \qquad \text{ and } \qquad H^{\bullet,\bullet}_{A}(X) \;:=\; \frac{\ker\del\delbar}{\imm\del+\imm\delbar} \;, $$
provide a kind of bridge connecting Dolbeault and de Rham cohomologies. In fact, the identity induces natural maps
$$
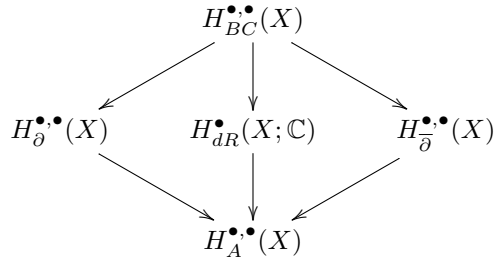
\xymatrix{
 & H^{\bullet,\bullet}_{BC}(X) \ar[d]\ar[ld]\ar[rd] & \\
 H^{\bullet,\bullet}_{\del}(X) \ar[rd] & H^{\bullet}_{dR}(X;\C) \ar[d] & H^{\bullet,\bullet}_{\delbar}(X) \ar[ld] \\
 & H^{\bullet,\bullet}_{A}(X) &
 }
$$
of (bi-)graded vector spaces.
Initially introduced in \cite{bott-chern} and \cite{aeppli}, they naturally arose in several problems in complex analysis, complex geometry, and theoretical physics.

\medskip

There are two complementary directions in studying Bott-Chern cohomology. On the one side, one can be interested in studying its relation with de Rham cohomology. The very special property that the natural map $\bigoplus_{p+q=\bullet} H^{p,q}_{BC}(X) \to H^{\bullet}_{dR}(X;\C)$ is injective, (and hence an isomorphism, \cite[Lemma 5.15, Remark 5.16, 5.21]{deligne-griffiths-morgan-sullivan},) is called {\em $\del\delbar$-Lemma property}, \cite{deligne-griffiths-morgan-sullivan}, and it holds for compact K\"ahler manifolds, \cite[Main Theorem]{deligne-griffiths-morgan-sullivan}. In this view, in \cite[Theorem A, Theorem B]{angella-tomassini-3}, we studied an inequality {\itshape à la} Fr\"olicher for Bott-Chern cohomology of compact complex manifolds of complex dimension $n$, which provides a characterization of $\del\delbar$-Lemma: for any $k\in\Z$,
\begin{eqnarray*}
\Delta^k &:=& \sum_{p+q=k} \left( \dim_\C H^{p,q}_{BC}(X) + \dim_\C H^{n-q,n-p}_{BC}(X) \right) - 2\, b_k \;\geq\; 0 \;,
\end{eqnarray*}
and the equality holds for any $k\in\Z$ if and only if $X$ satisfies the $\del\delbar$-Lemma. Note that, for compact complex surfaces, the cohomologically-K\"ahlerness expressed by the $\del\delbar$-Lemma property is in fact equivalent to K\"ahlerness, as a consequence of the K\"ahlerness characterization in terms of the first Betti number being even. More precisely, for compact complex surfaces, one has that the only degree of non-K\"ahlerness is $\Delta^2\in2\,\N$, \cite[Theorem 1.1]{angella-dloussky-tomassini}.

\medskip

On the other side, one can focus on the algebraic structure of $H^{\bullet,\bullet}_{BC}(X)$ in relation with $\wedge^{\bullet,\bullet}X$. In fact, $H^{\bullet,\bullet}_{BC}(X)$ has a structure of algebra, induced by the exterior product, and $H^{\bullet,\bullet}_{A}(X)$ has a structure of $H^{\bullet,\bullet}_{BC}(X)$-module.

Let us consider the de Rham cohomology $H^{\bullet}_{dR}(X;\R)$ of a compact differentiable manifold $X$. It has a structure of algebra given by the cap product, induced by the wedge product on the space $\wedge^\bullet X$ of differential forms. But, in general, it is not possible to choose a basis of representatives having a structure of algebra. In \cite[\S 3.1]{zhou}, the space of harmonic forms over a compact Riemannian manifold is endowed with a structure of $\mathrm{A}_{\infty}$-algebra in the sense of Stasheff, \cite[II, Definition 2.1]{stasheff-1-2}. The first linear maps of this structure are $m_1=0$ and $m_2=\left(\sspace \wedge \ssspace\right)^H$, where $\left(\sspace\right)^H$ denotes the projection onto the space of harmonic forms with respect to the fixed Riemannian metric. In \cite[Theorem 3.1, Corollary A.5]{lu-palmieri-wu-zhang}, it is proven that, for any $n\in\N\setminus\{0,1,2\}$, for any $x_1,\ldots,x_n$ harmonic forms, the element $m_n\left(x_1,\ldots,x_n\right)$ is a representative of the $n$th Massey product of $[x_1], \ldots, [x_n]$.

As firstly considered by Sullivan, \cite[page 326]{sullivan}, there is an ``incompatibility of wedge products and harmonicity of forms'' on a compact differentiable Riemannian manifold. This leads to the notion, introduced and studied by D. Kotschick, of {\em formal metric}, \cite[Definition 1]{kotschick}, namely, a Riemannian metric such that the space of harmonic forms has a structure of algebra. 
Note that, by definition, if the fixed Riemannian metric is formal in the sense of Kotschick, then, instead of $\left( \sspace \wedge \ssspace \right)^H$, one has just the wedge product between harmonic forms, letting the $A_{\infty}$ structure being in fact an algebra. The existence of such a metric is a stronger condition than formality: see, for example, the obstructions provided in \cite[Theorem 6, Theorem 7, Theorem 9]{kotschick}.

D. Sullivan introduced the notion of formality, see \cite[\S12]{sullivan}. It provides a way to relate the algebra structure of the complex of forms and the algebra structure of the de Rham cohomology.
A differentiable manifold $X$ is called {\em formal} in the sense of Sullivan \cite{sullivan} if the minimal model of its associated differential graded algebra $\left( \wedge^{\bullet}X,\, \de \right)$ and of the de Rham cohomology $\left( H^\bullet_{dR}(X;\R),\, 0 \right)$ coincide, \cite[page 315]{sullivan}.
In other words, we have a diagram
$$ \xymatrix{
  & \left( M^{\bullet},\, \de_M \right) \ar[ld]_{f} \ar[rd]^{g} & \\
 \left( \wedge^\bullet X,\, \de\right) & & \left( H^\bullet_{dR}(X;\R),\, 0 \right)
} $$
of differential $\Z$-graded algebras with $f$ and $g$ quasi-isomorphisms.
This means that the rational homotopy of $X$ ``can be computed formally from'' the cohomology ring $H^{\bullet}_{dR}(X;\R)$.
In particular, if $X$ is formal in the sense of Sullivan, \cite{sullivan}, then all the Massey products vanish.
A {\em geometrically formal} compact differentiable manifold, (that is, a compact differentiable manifold admitting a formal Riemannian metric,) is formal in the sense of Sullivan, \cite[\S1]{kotschick}.

\medskip

In this note, we are interested in notions of formality strictly related to complex structures. To this aim, recall the work by J. Neisendorfer and L. Taylor on Dolbeault formality, \cite{neisendorfer-taylor}. They study {\em Dolbeault formal} manifolds, that is, complex manifolds for which the double complex of differential forms is formal as a differential $\Z^2$-graded algebra with respect to the $\delbar$ operator. This happens, for example, for compact K\"ahler manifolds, or, more in general, for compact complex manifolds satisfying the $\del\delbar$-Lemma, \cite[Theorem 8]{neisendorfer-taylor}. See also \cite{tomassini-torelli}.

\medskip

But the above formality theory, valid for both de Rham and Dolbeault cohomology, seems not to work exactly for Bott-Chern cohomology, due to technical reasons. For example, note that the Bott-Chern cohomology is not computed as the cohomology of a differential graded algebra. We try to start here the ``algebraic'' study of the Bott-Chern cohomology by investigating the notion of {\em geometrically-Bott-Chern-formal metrics} on compact complex manifolds, namely, Hermitian metrics for which the product of harmonic forms with respect to the Bott-Chern Laplacian, \cite{kodaira-spencer-3, schweitzer}, is still harmonic. (Here, harmonic is intented in the sense of the Hodge theory developed by M. Schweitzer for Bott-Chern and Aeppli cohomologies in \cite{schweitzer}.) This seems, as far as we know, a first attempt to understand a possible notion of formality related to Bott-Chern cohomology.

More precisely, we introduce {\em triple Aeppli-Bott-Chern-Massey products}, see \S\ref{subsec:massey}. We show that they provide an obstruction to the existence of geometrically-Bott-Chern-formal metrics.

\renewcommand{\referenza}{\ref{thm:ABC-massey-geom-BC}}
\begin{thm*}
 Triple Aeppli-Bott-Chern-Massey products vanish on compact complex geometrically-Bott-Chern-formal manifolds.
\end{thm*}

Several examples are studied explicitly in Section \ref{sec:examples}, here including Iwasawa manifold, compact complex surfaces diffeomorphic to solvmanifolds, Calabi-Eckmann structures on $\mathbb{S}^1\times\mathbb{S}^3$ and on $\mathbb{S}^3\times\mathbb{S}^3$, holomorphically-parallelizable Nakamura manifold.

\bigskip

\noindent {\itshape Acknowledgments.} The authors would like to thank Hisashi Kasuya, Sara Torelli, Nicoletta Tardini, Jos\'{e} Ignacio Cogolludo Agust\'{i}n, Georges Dloussky, Giovanni Bazzoni for several discussions. Many thanks also to the anonymous Referee for her/his useful suggestions.

\section{Bott-Chern cohomology and formality}
Let $X$ be a compact complex manifold of complex dimension $n$. We first recall the definition of Bott-Chern cohomology and some results on Hodge theory. Then we introduce the notion of geometrically-Bott-Chern-cohomology.

\subsection{Bott-Chern cohomology}
The {\em Bott-Chern cohomology}, \cite{bott-chern}, is the $\Z^2$-graded algebra
$$ H^{\bullet,\bullet}_{BC}(X) \;:=\; H^{\bullet,\bullet}_{BC}\left(\wedge^{\bullet,\bullet}X,\, \del,\, \delbar\right) \;:=\; \frac{\ker\del\cap\ker\delbar}{\imm\del\delbar} \;. $$
The {\em Aeppli cohomology}, \cite{aeppli}, is the $\Z^2$-graded $H_{BC}(X)$-module
$$ H^{\bullet,\bullet}_{A}(X) \;:=\; H^{\bullet,\bullet}_{A}\left(\wedge^{\bullet,\bullet}X,\, \del,\, \delbar\right) \;:=\; \frac{\ker\del\delbar}{\imm\del+\imm\delbar} \;. $$

\medskip

For a given Hermitian metric $g$ on $X$, consider the $4$th order self-adjoint elliptic differential operators, \cite[Proposition 5]{kodaira-spencer-3}, \cite[\S2.b]{schweitzer},
$$ \Delta_{BC}^{g} \;:=\;
      \left(\del\delbar\right)\left(\del\delbar\right)^* + \left(\del\delbar\right)^*\left(\del\delbar\right)
      + \left(\delbar^*\del\right)\left(\delbar^*\del\right)^* + \left(\delbar^*\del\right)^*\left(\delbar^*\del\right)
      + \delbar^*\delbar + \del^*\del \;,
$$
and, \cite[\S2.c]{schweitzer},
$$ \Delta_{A}^{g} \;:=\;
      \del\del^* + \delbar\delbar^*
      + \left(\del\delbar\right)^*\left(\del\delbar\right) + \left(\del\delbar\right)\left(\del\delbar\right)^*
      + \left(\delbar\del^*\right)^*\left(\delbar\del^*\right) + \left(\delbar\del^*\right)\left(\delbar\del^*\right)^* \;.
$$
Consider the inclusions
$$
\iota_{BC}^{g} \colon \ker\Delta_{BC}^{g} \hookrightarrow \wedge^{\bullet,\bullet} X
\qquad \text{ and } \qquad
\iota_{A}^{g} \colon \ker\Delta_{A}^{g} \hookrightarrow \wedge^{\bullet,\bullet} X \;.
$$
By \cite[Corollaire 2.3, \S2.c]{schweitzer}, one has that
$$ H_{BC}\left(\iota_{BC}^{g}\right) \text{ and } H_{A}\left(\iota_{A}^{g}\right) \text{ are isomorphisms of vector spaces} \;. $$

\medskip

For a Hermitian metric $g$, the $\C$-linear Hodge-$*$-operator $*_g \colon \wedge^{\bullet_1,\bullet_2}X \to \wedge^{n-\bullet_2,n-\bullet_1}X$ induces the isomorphism
$$ *_g \colon H^{\bullet_1,\bullet_2}_{BC}(X) \stackrel{\simeq}{\to} H^{n-\bullet_2,n-\bullet_1}_{A}(X) \;. $$

\subsection{Geometrically-Bott-Chern-formality}

Inspired by \cite{kotschick}, we consider the following notion, concerning Hermitian metrics for which $\ker\Delta_{BC}^{g}$ has a structure of algebra.

\begin{defi}
 A Hermitian metric $g$ on a compact complex manifold $X$ is called {\em geometrically-Bott-Chern-formal} if it yields the inclusion $\iota_{BC}^{g} \colon \left( \ker\Delta_{BC}^{g},\, 0,\, 0 \right) \hookrightarrow \left( \wedge^{\bullet,\bullet}X,\, \del,\, \delbar \right)$ of bi-differential $\Z^2$-graded algebras, such that $H_{BC}(\iota_{BC}^{g})$ is an isomorphism.
\end{defi}

\begin{defi}
 A compact complex manifold $X$ is called {\em geometrically-Bott-Chern-formal} if there exists a Hermitian metric $g$ on $X$ being geometrically-Bott-Chern-formal.
\end{defi}

\section{Aeppli-Bott-Chern-Massey products}
In this section, we define Aeppli-Bott-Chern-Massey products on complex manifolds, and we show that they provide obstructions to geometrically-Bott-Chern-formality.

\subsection{Aeppli-Bott-Chern-Massey products on complex manifolds}\label{subsec:massey}

We recall that, on a compact complex manifold $X$, the Bott-Chern cohomology $H^{\bullet,\bullet}_{BC}(X)$ has a structure of algebra. The Aeppli cohomology $H^{\bullet,\bullet}_{A}(X)$ has a structure of $H^{\bullet,\bullet}_{BC}(X)$-module. We start by giving the following definition and by proving its coherency.

\begin{defi}
 Let $\left( A^{\bullet,\bullet},\, \del,\, \delbar \right)$ be a bi-differential $\Z^2$-graded algebra. Take
 $$
 \mathfrak{a}_{12} \;=\; \left[\alpha_{12}\right] \in H^{p,q}_{BC}(A^{\bullet,\bullet}) \;,
 \qquad
 \mathfrak{a}_{23} \;=\; \left[\alpha_{23}\right] \in H^{r,s}_{BC}(A^{\bullet,\bullet}) \;,
 $$
 $$
\mathfrak{a}_{34} \;=\; \left[\alpha_{34}\right] \in H^{u,v}_{BC}(A^{\bullet,\bullet}) \;,
 $$
 such that $\mathfrak{a}_{12} \cp \mathfrak{a}_{23}=0$ in $H^{p+r,q+s}_{BC}(A^{\bullet,\bullet})$ and $\mathfrak{a}_{23} \cp \mathfrak{a}_{34}=0$ in $H^{r+u,s+v}_{BC}(A^{\bullet,\bullet})$: let
 $$
 (-1)^{p+q} \, \alpha_{12} \wedge \alpha_{23} \;=\; \del\delbar \alpha_{13}
 \qquad \text{ and } \qquad
 (-1)^{r+s} \, \alpha_{23} \wedge \alpha_{34} \;=\; \del\delbar \alpha_{24} \;.
 $$
 The {\em triple Aeppli-Bott-Chern-Massey product} $\left\langle \mathfrak{a}_{12},\, \mathfrak{a}_{23},\, \mathfrak{a}_{34} \right\rangle_{ABC}$ is defined as
 \begin{eqnarray*}
 \lefteqn{\left\langle \mathfrak{a}_{12},\, \mathfrak{a}_{23},\, \mathfrak{a}_{34} \right\rangle_{ABC}} \\[5pt]
 &:=&
 \left[ (-1)^{p+q}\, \alpha_{12} \wedge \alpha_{24} - (-1)^{r+s}\, \alpha_{13} \wedge \alpha_{34} \right] \\[5pt]
 &\in& \frac{H^{p+r+u-1, q+s+v-1}_{A}(A^{\bullet,\bullet})}{\mathfrak{a}_{12} \cp H^{r+u-1, s+v-1}_{A}(A^{\bullet,\bullet}) + H^{p+r-1, q+s-1}_{A}(A^{\bullet,\bullet}) \cp \mathfrak{a}_{34}} \;.
 \end{eqnarray*}
\end{defi}

\begin{proof}
We have to prove that the form $(-1)^{p+q}\, \alpha_{12} \wedge \alpha_{24} - (-1)^{r+s}\, \alpha_{13} \wedge \alpha_{34}$ is $\del\delbar$-closed, and that its class in the quotient depends neither on the chosen representatives of $\mathfrak{a}_{12}$, $\mathfrak{a}_{23}$,  $\mathfrak{a}_{34}$, nor on the chosen elements $\alpha_{13}$, $\alpha_{24}$.

\medskip

First of all, note that, since $\alpha_{12}$ and $\alpha_{34}$ are $\del$-closed and $\delbar$-closed, then
\begin{eqnarray*}
 \lefteqn{\del\delbar \left( (-1)^{p+q}\, \alpha_{12} \wedge \alpha_{24} - (-1)^{r+s}\, \alpha_{13} \wedge \alpha_{34} \right)}\\[5pt]
 &=& (-1)^{p+q}\, \alpha_{12} \wedge \del\delbar \alpha_{24} - (-1)^{r+s}\, \del\delbar \alpha_{13} \wedge \alpha_{34} \\[5pt]
 &=& (-1)^{p+q+r+s}\, \alpha_{12} \wedge \alpha_{23} \wedge \alpha_{34} - (-1)^{p+q+r+s}\, \alpha_{12} \wedge \alpha_{23} \wedge \alpha_{34} \\[5pt]
 &=& 0 \;.
\end{eqnarray*}
Hence the form $(-1)^{p+q}\, \alpha_{12} \wedge \alpha_{24} - (-1)^{r+s}\, \alpha_{13} \wedge \alpha_{34}$ defines a class in $H^{p+r+u-1,q+s+v-1}_{A}(A^{\bullet,\bullet})$.

\medskip

Now, suppose, for example, that $\mathfrak{a}_{12}=0$ in $H^{p,q}_{BC}(A^{\bullet,\bullet})$: let $\alpha_{12}=\del\delbar \xi$. Then we can take $\alpha_{13} = (-1)^{p+q}\, \xi \wedge \alpha_{23}$. Therefore we get
\begin{eqnarray*}
 (-1)^{p+q}\, \alpha_{12} \wedge \alpha_{24} - (-1)^{r+s}\, \alpha_{13} \wedge \alpha_{34}
 &=& (-1)^{p+q}\, \del\delbar \xi \wedge \alpha_{24} - (-1)^{p+q}\, \xi \wedge \del\delbar \alpha_{24} \\[5pt]
 &=& \del \left( (-1)^{p+q}\, \delbar\xi \wedge \alpha_{24} \right) + \delbar \left( \xi \wedge \del \alpha_{24} \right) \;,
\end{eqnarray*}
which is a null term in the Aeppli cohomology.

\medskip

Finally, suppose, for example, that $\alpha_{13} \in \ker\del + \ker\delbar$. Then the form $\alpha_{13}$ defines a class in $H^{p+r-1,q+s-1}_{A}(A^{\bullet,\bullet})$, and the term $\alpha_{13} \wedge \alpha_{34}$ is null in the quotient $H^{p+r+u-1, q+s+v-1}_{A}(A^{\bullet,\bullet}) \slash \big( \mathfrak{a}_{12} \cp H^{r+u-1, s+v-1}_{A}(A^{\bullet,\bullet}) + H^{p+r-1, q+s-1}_{A}(A^{\bullet,\bullet}) \cp \mathfrak{a}_{34} \big)$.
\end{proof}

Note that, the triple Aeppli-Bott-Chern Massey products belonging to a quotient of the Aeppli cohomology, one can not define higher order products straightforwardly as usual.

\begin{defi}
 Let $X$ be a complex manifold. The {\em triple Aeppli-Bott-Chern Massey products}, (shortly, {\em Aeppli-Bott-Chern-Massey products},) are defined as the triple Aeppli-Bott-Chern-Massey products of the bi-differential $\Z^2$-graded algebra $\left( \wedge^{\bullet,\bullet}X,\, \del,\, \delbar \right)$.
\end{defi}

\subsection{Aeppli-Bott-Chern-Massey products and formality}
We prove that triple Aeppli-Bott-Chern-Massey products provides obstructions to geometrically-Bott-Chern-formality for compact complex manifolds.

\medskip

The following lemma is straightforward.

\begin{lemma}\label{lemma:abc-massey-morphism}
 Let $f \colon \left( A^{\bullet,\bullet},\, \del_A,\, \delbar_A \right) \to \left( B^{\bullet,\bullet},\, \del_B,\, \delbar_B \right)$ be a morphism of bi-differential $\Z^2$-graded algebras. For any $\mathfrak{a}_{12} \in H^{p,q}_{BC}(A^{\bullet,\bullet})$, $\mathfrak{a}_{23} \in H^{r,s}_{BC}(A^{\bullet,\bullet})$, $\mathfrak{a}_{34} \in H^{u,v}_{BC}(A^{\bullet,\bullet})$ such that $\mathfrak{a}_{12} \cp \mathfrak{a}_{23} = 0$ and $\mathfrak{a}_{23} \cp \mathfrak{a}_{34} = 0$, it holds
 \begin{eqnarray*}
 \lefteqn{ \left(H_A(f)\right) \left( \left\langle \mathfrak{a}_{12},\, \mathfrak{a}_{23},\, \mathfrak{a}_{34} \right\rangle_{ABC} \right) } \\[5pt]
 &=& \left\langle \left(H_{BC}(f)\right)(\mathfrak{a}_{12}),\, \left(H_{BC}(f)\right)(\mathfrak{a}_{23}),\, \left(H_{BC}(f)\right)(\mathfrak{a}_{34}) \right\rangle_{ABC} \;.
 \end{eqnarray*}
\end{lemma}

\begin{proof}
 Consider $\mathfrak{a}_{12}=\left[\alpha_{12}\right]$, $\mathfrak{a}_{23}=\left[\alpha_{23}\right]$, $\mathfrak{a}_{34}=\left[\alpha_{34}\right]$, and $(-1)^{p+q}\,\alpha_{12}\wedge\alpha_{23}=\del\delbar\alpha_{13}$, $(-1)^{r+s}\,\alpha_{23}\wedge\alpha_{34}=\del\delbar\alpha_{24}$. We have
 \begin{eqnarray*}
  \lefteqn{ \left(H_{A}(f)\right) \left( \left\langle \mathfrak{a}_{12},\, \mathfrak{a}_{23},\, \mathfrak{a}_{34} \right\rangle_{ABC} \right) }\\[5pt]
  &=& \left(H_{A}(f)\right) \left(\left[ (-1)^{p+q}\, \alpha_{12}\wedge\alpha_{24} - (-1)^{r+s}\, \alpha_{13}\wedge\alpha_{34} \right]\right) \\[5pt]
  &=& \left[ (-1)^{p+q}\, f(\alpha_{12})\wedge f(\alpha_{24}) - (-1)^{r+s}\, f(\alpha_{13})\wedge f(\alpha_{34}) \right] \\[5pt]
  &=& \left\langle \left[f(\alpha_{12})\right],\, \left[f(\alpha_{23})\right],\, \left[f(\alpha_{34})\right] \right\rangle_{ABC} \\[5pt]
  &=& \left\langle \left(H_{BC}(f)\right)(\mathfrak{a}_{12}),\, \left(H_{BC}(f)\right)(\mathfrak{a}_{23}),\, \left(H_{BC}(f)\right)(\mathfrak{a}_{34}) \right\rangle_{ABC} \;.
 \end{eqnarray*}
 Indeed, note that $(-1)^{p+q}\, \left(H_{BC}(f)\right)(\mathfrak{a}_{12}) \cp \left(H_{BC}(f)\right)(\mathfrak{a}_{23}) = (-1)^{p+q}\, \left[f(\alpha_{12})\right] \cp \left[f(\alpha_{23})\right] = (-1)^{p+q}\, \left[f(\alpha_{12}\wedge\alpha_{23})\right] = \left[\del\delbar f(\alpha_{13})\right]$ and $(-1)^{r+s}\, \left(H_{BC}(f)\right)(\mathfrak{a}_{23}) \cp \left(H_{BC}(f)\right)(\mathfrak{a}_{34}) = \left[\del\delbar f(\alpha_{24})\right]$.
\end{proof}

We now apply the above lemma to the case of geometrically-Bott-Chern-formal manifolds.

\begin{thm}\label{thm:ABC-massey-geom-BC}
 Triple Aeppli-Bott-Chern-Massey products vanish on compact complex geometrically-Bott-Chern-formal manifolds.
\end{thm}

\begin{proof}
 Let $g$ be a geometrically-Bott-Chern-formal metric on the compact complex manifold $X$. In particular, $\ker \Delta_{BC}^{g}$ has a structure of algebra and the inclusion $\iota_{BC}^{g} \colon \left( \ker\Delta_{BC}^{g} ,\, 0 ,\, 0 \right) \to \left( \wedge^{\bullet,\bullet}X ,\, \del ,\, \delbar \right)$ is a morphism of algebras yielding the isomorphism $H_{BC}\left(\iota_{BC}^{g}\right)$.
 
 We are hence reduced to prove the following claim. Let $X$ be a compact complex manifold. Suppose that there exists a morphism $f \colon \left( M^{\bullet,\bullet},\, 0,\, 0 \right) \to \left( \wedge^{\bullet,\bullet}X,\, \del,\, \delbar \right)$ of bi-differential $\Z^2$-graded algebras such that $H_{BC}(f)$ is an isomorphism. Then the triple Aeppli-Bott-Chern-Massey products of $X$ vanish.
 
 Indeed, for any $\mathfrak{a}_{12} \in H^{p,q}_{BC}(X)$, $\mathfrak{a}_{23} \in H^{r,s}_{BC}(X)$, $\mathfrak{a}_{34} \in H^{u,v}_{BC}(X)$ such that $\mathfrak{a}_{12} \cp \mathfrak{a}_{23} = 0$ and $\mathfrak{a}_{23} \cp \mathfrak{a}_{34} = 0$, by Lemma \ref{lemma:abc-massey-morphism} one has
 \begin{eqnarray*}
  \left\langle \mathfrak{a}_{12}, \mathfrak{a}_{23}, \mathfrak{a}_{34} \right\rangle_{ABC} &=& \left( H_{A}(f) \right) \left\langle H_{BC}(f)^{-1} \left(\mathfrak{a}_{12}\right), H_{BC}(f)^{-1} \left(\mathfrak{a}_{23}\right), H_{BC}(f)^{-1} \left(\mathfrak{a}_{34}\right) \right\rangle \\[5pt]
  &=& \left( H_{A}(f) \right) (0) \;=\; 0 \;.
 \end{eqnarray*}
 Indeed, note that triple Aeppli-Bott-Chern-Massey products vanish in $\Z^2$-graded algebras with zero differentials.
\end{proof}

\section{Examples}\label{sec:examples}
In this section, we provide some examples of (non-)geometrically-Bott-Chern-formal manifolds.

\subsection{Iwasawa manifold}
We show a non-zero Aeppli-Bott-Chern-Massey product on the Iwasawa manifold, which is one of the simplest example of non-K\"ahler complex nilmanifolds. 
Hence, by Theorem \ref{thm:ABC-massey-geom-BC}, we get the following.

\begin{prop}
The Iwasawa manifold is not geometrically-Bott-Chern-formal.
\end{prop}

\begin{proof}
The {\em Iwasawa manifold} is $\mathbb{I}_3 := \left. \mathbb{H}\left(3;\Z\left[\im\right]\right) \right\backslash \mathbb{H}(3;\C)$, where the connected simply-connected complex $2$-step nilpotent Lie group
$$
\mathbb{H}(3;\mathbb{C}) := \left\{
\left(
\begin{array}{ccc}
 1 & z^1 & z^3 \\
 0 &  1  & z^2 \\
 0 &  0  &  1
\end{array}
\right) \in \mathrm{GL}(3;\mathbb{C}) \st z^1,\,z^2,\,z^3 \in\C \right\}
\;,
$$
endowed with the product induced by matrix multiplication, is the $3$-dimensional Heisenberg group over $\mathbb{C}$, and $\mathbb{H}\left(3;\Z\left[\im\right]\right):=\mathbb{H}(3;\C)\cap\GL\left(3;\Z\left[\im\right]\right)$ is a lattice in $\mathbb{H}(3;\C)$.

One gets that $\mathbb{I}_3$ is a $3$-dimensional holomorphically-parallelizable complex nilmanifold, endowed with a $\mathbb{H}(3;\C)$-left-invariant complex structure inherited by the standard complex structure on $\mathbb{H}(3;\mathbb{C})$.

A $\mathbb{H}(3;\C)$-left-invariant co-frame for the space of $(1,0)$-forms on $\mathbb{I}_3$ is given by
$$
\left\{
\begin{array}{rcl}
 \varphi^1 &:=& \de z^1 \\[5pt]
 \varphi^2 &:=& \de z^2 \\[5pt]
 \varphi^3 &:=& \de z^3-z^1\,\de z^2
\end{array}
\right. \;,
$$
and the corresponding structure equations are
$$
\left\{
\begin{array}{rcl}
 \de\varphi^1 &=& 0 \\[5pt]
 \de\varphi^2 &=& 0 \\[5pt]
 \de\varphi^3 &=& -\varphi^1\wedge\varphi^2
\end{array}
\right. \;.
$$

\medskip

Consider
 $$
 \mathfrak{a}_{12} \;:=\; \left[ \varphi^1 \wedge \varphi^2 \right] \;\in\; H^{2,0}_{BC}(\mathbb{I}_3) \;,
 \qquad
 \mathfrak{a}_{23} \;:=\; \left[ \bar\varphi^1 \wedge \bar\varphi^2 \right] \;\in\; H^{0,2}_{BC}(\mathbb{I}_3) \;,
 $$
 $$
 \mathfrak{a}_{34} \;:=\; \left[ \bar\varphi^1 \right] \;\in\; H^{0,1}_{BC}(\mathbb{I}_3) \;.
 $$
Since $\del\delbar \left( - \varphi^3 \wedge \bar\varphi^3 \right) = \varphi^1 \wedge \varphi^2 \wedge \bar\varphi^1 \wedge \bar\varphi^2$, we take
 $$ \alpha_{13} \;=\; - \varphi^3 \wedge \bar\varphi^3 \qquad \text{ and } \qquad \alpha_{24} \;=\; 0 \;. $$
By noting that, \cite[\S1.c]{schweitzer},
 \begin{eqnarray*}
 \lefteqn{\frac{H^{1, 2}_{A}(X)}{H^{1,1}_{A}(X) \cp \mathfrak{a}_{34}} }\\[5pt]
 &=& \left\langle \left[ \varphi^{1}\wedge\bar\varphi^{2}\wedge\bar\varphi^{3} \right] ,\; \left[ \varphi^{2}\wedge\bar\varphi^{2}\wedge\bar\varphi^{3} \right] ,\; \left[ \varphi^3 \wedge \bar\varphi^1 \wedge \bar\varphi^3 \right] ,\; \left[ \varphi^3 \wedge \bar\varphi^2 \wedge \bar\varphi^3 \right] \right\rangle \;, \end{eqnarray*}
we get that
 $$ \mathfrak{a}_{1234} \;:=\; \left\langle \mathfrak{a}_{12},\, \mathfrak{a}_{23},\, \mathfrak{a}_{34} \right\rangle_{ABC} \;=\; \left[ \alpha_{12}\wedge\alpha_{24}-\alpha_{13}\wedge\alpha_{34} \right] \;=\; \left[ \varphi^3 \wedge \bar\varphi^1 \wedge \bar\varphi^3 \right] $$
is a non-trivial Aeppli-Bott-Chern-Massey product.
\end{proof}

\subsection{Compact complex surfaces}
Compact complex non-K\"ahler surfaces diffeomorphic to solvmanifolds are studied by K. Hasegawa in \cite{hasegawa-jsg}.
They are Inoue surface of type $\mathcal{S}_M$, primary Kodaira surface, secondary Kodaira surface, and Inoue surface of type $\mathcal{S}^\pm$, and are endowed with left-invariant complex structures.
In \cite{angella-dloussky-tomassini}, their Dolbeault and Bott-Chern cohomology is studied.
It turns out that the Dolbeault and Bott-Chern cohomologies of such manifolds can be fully recovered by the sub-double-complex of left-invariant forms, see \cite[Theorem 4.1]{angella-dloussky-tomassini}. More precisely, if the compact complex surface $X$ is diffeomorphic to the solvmanifold $\left. \Gamma \middle\backslash G \right.$, then the inclusion $\iota \colon \left( \wedge^{\bullet,\bullet}\mathfrak{g}^*,\, \del,\, \delbar \right) \to \left( \wedge^{\bullet,\bullet}X,\, \del,\, \delbar \right)$ yields the isomorphisms $H_{\delbar}(\iota)$ and $H_{BC}(\iota)$, where $\mathfrak{g}$ denotes the Lie algebra associated to $G$.

\medskip

We show here that compact complex surfaces diffeomorphic to solvmanifolds are geometrically-Bott-Chern-formal.
(Furthermore, except in the case of primary Kodaira surface, they are also geometrically-Dolbeault-formal.
Here, by geometrically-Dolbeault-formal, we mean that $X$ admits a Hermitian metric with respect to which $\ker\Delta^{g}_{\delbar}$ is an algebra, where $\Delta^{g}_{\delbar}:=\left[\delbar,\delbar^*\right]$ is the Dolbeault Laplacian. As for the case of a primary Kodaira surface, it is not geometrically-Dolbeault-formal. This follows from \cite[Proposition 4.2]{cordero-fernandez-ugarte}, stating that any nilmanifold with first Betti number $b_1 = 3$ endowed with a nilpotent complex structure has a nonzero triple Dolbeault-Massey product.)

\begin{prop}\label{prop:formality-surfaces}
 Let $X$ be either an Inoue surface of type $\mathcal{S}_M$, or a primary Kodaira surface, or an Inoue surface of type $\mathcal{S}^{\pm}$, or a secondary Kodaira surface. Then $X$ is geometrically-Bott-Chern-formal.
\end{prop}

\begin{proof}
Explicit representatives for the cohomologies are provided in \cite[Table 1, Table 2]{angella-dloussky-tomassini}. More precisely, a left-invariant Hermitian metric $g$ is fixed.
The harmonic representatives with respect to such a metric, and in terms of a left-invariant orthonormal co-frame $\left\{ \varphi^1, \varphi^2 \right\}$ for the holomorphic co-tangent bundle, are the following:
\begin{itemize}
 \item for Inoue surface of type $\mathcal{S}_M$:
  $$ \ker\Delta_{\delbar}^{g} \;=\; \wedge^{\bullet,\bullet} \left( \C\left\langle 1 \right\rangle \oplus \C\left\langle \bar\varphi^{2} \right\rangle \oplus \C\left\langle \varphi^1\wedge\varphi^2\wedge\bar\varphi^1 \right\rangle \right) \;, $$
  and
  \begin{eqnarray*}
   \ker\Delta_{BC}^{g} &=& \wedge^{\bullet,\bullet} \left( \C\left\langle 1 \right\rangle \oplus \C\left\langle \varphi^2\wedge\bar\varphi^{2} \right\rangle \oplus \C\left\langle \varphi^1\wedge\varphi^2\wedge\bar\varphi^1 \right\rangle \right. \\[5pt]
   && \left. \oplus \C\left\langle \varphi^1\wedge\bar\varphi^1\wedge\bar\varphi^2 \right\rangle \oplus \C\left\langle \varphi^1\wedge\varphi^2\wedge\bar\varphi^1\wedge\bar\varphi^2 \right\rangle \right) \;;
  \end{eqnarray*}

 \item for primary Kodaira surface:
  \begin{eqnarray*}
   \ker\Delta_{\delbar}^{g}
    &=& \C\left\langle 1 \right\rangle \oplus \C\left\langle \varphi^1 \right\rangle \oplus \C\left\langle \bar\varphi^1,\, \bar\varphi^2 \right\rangle \\[5pt]
    && \oplus \C\left\langle \varphi^1\wedge\varphi^2 \right\rangle \oplus \C\left\langle \varphi^1\wedge\bar\varphi^2,\, \varphi^2\wedge\bar\varphi^1 \right\rangle \\[5pt]
    && \oplus \C\left\langle \bar\varphi^1\wedge\bar\varphi^2 \right\rangle \oplus \C\left\langle \varphi^1\wedge\varphi^2\wedge\bar\varphi^1,\, \varphi^1\wedge\varphi^2\wedge\bar\varphi^2 \right\rangle \\[5pt]
    && \oplus \C\left\langle \varphi^2\wedge\bar\varphi^2\wedge\bar\varphi^2 \right\rangle \oplus \C\left\langle \varphi^1\wedge\varphi^2\wedge\bar\varphi^1\wedge\bar\varphi^2 \right\rangle \;,
  \end{eqnarray*}
  and
  \begin{eqnarray*}
   \ker\Delta_{BC}^{g} &=& \wedge^{\bullet,\bullet} \left( \C\left\langle 1 \right\rangle \oplus \C\left\langle \varphi^1 \right\rangle \oplus \C\left\langle \bar\varphi^1 \right\rangle \right. \\[5pt]
    && \left. \oplus \C\left\langle \varphi^1\wedge\varphi^2 \right\rangle \oplus \C\left\langle \varphi^1\wedge\bar\varphi^2,\, \varphi^2\wedge\bar\varphi^1 \right\rangle \right. \\[5pt]
   && \left. \oplus \C\left\langle \bar\varphi^1\wedge\bar\varphi^2 \right\rangle \oplus \C\left\langle \varphi^1\wedge\varphi^2\wedge\bar\varphi^2 \right\rangle \right. \\[5pt]
    && \left. \oplus \C\left\langle \varphi^2\wedge\bar\varphi^1\wedge\bar\varphi^2 \right\rangle \right) \;;
  \end{eqnarray*}
  note in particular that
  $$ \varphi^1\in\ker\Delta_{\delbar}^{g} \text{ and } \bar\varphi^1\in\ker\Delta_{\delbar}^{g} \quad \text{ but } \quad \varphi^1\wedge\bar\varphi^1 \not\in\ker\Delta_{\delbar}^{g} \;; $$

 \item for secondary Kodaira surface:
  $$ \ker\Delta_{\delbar}^{g} \;=\; \wedge^{\bullet,\bullet} \left( \C\left\langle 1 \right\rangle \oplus \C\left\langle \bar\varphi^{2} \right\rangle \oplus \C\left\langle \varphi^1\wedge\varphi^2\wedge\bar\varphi^1 \right\rangle \right) \;, $$
  and
  \begin{eqnarray*}
   \ker\Delta_{BC}^{g} &=& \wedge^{\bullet,\bullet} \left( \C\left\langle 1 \right\rangle \oplus \C\left\langle \varphi^1\wedge\bar\varphi^{1} \right\rangle \oplus \C\left\langle \varphi^1\wedge\varphi^2\wedge\bar\varphi^1 \right\rangle \right. \\[5pt]
   && \left. \oplus \C\left\langle \varphi^1\wedge\bar\varphi^1\wedge\bar\varphi^2 \right\rangle \oplus \C\left\langle \varphi^1\wedge\varphi^2\wedge\bar\varphi^1\wedge\bar\varphi^2 \right\rangle \right) \;;
  \end{eqnarray*}

 \item for Inoue surface of type $\mathcal{S}^\pm$:
  $$ \ker\Delta_{\delbar}^{g} \;=\; \wedge^{\bullet,\bullet} \left( \C\left\langle 1 \right\rangle \oplus \C\left\langle \bar\varphi^{2} \right\rangle \oplus \C\left\langle \varphi^1\wedge\varphi^2\wedge\bar\varphi^1 \right\rangle \right) \;, $$
  and
  \begin{eqnarray*}
   \ker\Delta_{BC}^{g} &=& \wedge^{\bullet,\bullet} \left( \C\left\langle 1 \right\rangle \oplus \C\left\langle \varphi^2\wedge\bar\varphi^{2} \right\rangle \oplus \C\left\langle \varphi^1\wedge\varphi^2\wedge\bar\varphi^1 \right\rangle \right. \\[5pt]
   && \left. \oplus \C\left\langle \varphi^1\wedge\bar\varphi^1\wedge\bar\varphi^2 \right\rangle \oplus \C\left\langle \varphi^1\wedge\varphi^2\wedge\bar\varphi^1\wedge\bar\varphi^2 \right\rangle \right) \;.
  \end{eqnarray*}
\end{itemize}
\end{proof}

\medskip

As an explicit example of a surface in class $\text{VII}$, in \cite[\S2.1]{angella-dloussky-tomassini}, the Calabi-Eckmann structure \cite{calabi-eckmann} on the Hopf surface is studied.
From the computations there, one has the following result.

\begin{thm}
 Let $X$ be the Hopf surface endowed with the Calabi-Eckmann complex structure. Then $X$ is both geometrically-Dolbeault-formal and geometrically-Bott-Chern-formal.
\end{thm}

\begin{proof}
In fact, see \cite[Proposition 2.4]{angella-dloussky-tomassini}, for a fixed Hermitian metric $g$ and a fixed left-invariant co-frame of the holomorphic cotangent bundle, one computes the following harmonic representatives:
$$ \ker\Delta^{g}_{\delbar} \;=\; \wedge^{\bullet,\bullet} \left( \C\left\langle 1 \right\rangle \oplus \C\left\langle \bar\varphi^{2} \right\rangle \oplus \C\left\langle \varphi^1\wedge\varphi^2\wedge\bar\varphi^1 \right\rangle \oplus \C\left\langle \varphi^1\wedge\varphi^2\wedge\bar\varphi^1\wedge\bar\varphi^2 \right\rangle \right) \;, $$
and
\begin{eqnarray*}
 \ker\Delta^{g}_{BC} &=& \wedge^{\bullet,\bullet} \left( \C\left\langle 1 \right\rangle \oplus \C\left\langle \bar\varphi^{1}\wedge\bar\varphi^{1} \right\rangle \oplus \C\left\langle \varphi^1\wedge\varphi^2\wedge\bar\varphi^1 \right\rangle \right. \\[5pt]
 && \left. \oplus \C\left\langle \varphi^1\wedge\bar\varphi^1\wedge\bar\varphi^2 \right\rangle \oplus \C\ \left\langle \varphi^1\wedge\varphi^2\wedge\bar\varphi^1\wedge\bar\varphi^2 \right\rangle \right) \;, 
\end{eqnarray*}
proving the statement.
\end{proof}

\subsection{Calabi-Eckmann structure on \texorpdfstring{$\mathbb{S}^3\times\mathbb{S}^3$}{S3 cross S3}}
E. Calabi and B. Eckmann constructed in \cite{calabi-eckmann} a complex structure on the manifolds $M_{u,v}$ diffeomorphic to $\mathbb{S}^{2u+1}\times\mathbb{S}^{2v+1}$, where $u,v\in\N$. More precisely, $M_{u,v}$ is the total space of an analytic fibre bundle with fibre a torus and base $\CP^u\times\CP^v$. It is also the total space of an analytic fibre bundle with fibre $M_{0,v}$ and base $\CP^v$. In the case $uv=0$, the manifold $M_{u,v}$ is called {\em Hopf manifold}.

\medskip

We show the following, concerning $\mathbb{S}^3\times\mathbb{S}^3$.

\begin{prop}
 The manifold $\mathbb{S}^3\times\mathbb{S}^3$ endowed with the Calabi-Eckmann complex structure is geometrically-Bott-Chern-formal.
\end{prop}

\begin{proof}
 Consider the differentiable manifold $X:=\mathbb{S}^3\times\mathbb{S}^3$. View $\mathbb{S}^3=SU(2)$ as a Lie group: it has a global left-invariant co-frame $\left\{e^1,e^2,e^3\right\}$ such that
 $$ \de e^1 \;=\; -2e^2\wedge e^3 \;, \qquad \de e^2 \;=\; 2e^1\wedge e^3 \;, \qquad \text{ and } \qquad \de e^3 \;=\; -2 e^1\wedge e^2 \;.$$
 Hence, we consider a global left-invariant co-frame $\left\{e^1,e^2,e^3,f^1,f^2,f^3\right\}$ for $TX$ with structure equations
 $$ \left\{\begin{array}{rcl}
            \de e^1 &=&          -  2\, e^2 \wedge e^3 \\[5pt]
            \de e^2 &=& \phantom{+} 2\, e^1 \wedge e^3 \\[5pt]
            \de e^3 &=&          -  2\, e^1 \wedge e^2 \\[5pt]
            \de f^1 &=&          -  2\, f^2 \wedge f^3 \\[5pt]
            \de f^2 &=& \phantom{+} 2\, f^1 \wedge f^3 \\[5pt]
            \de f^3 &=&          -  2\, f^1 \wedge f^2
           \end{array}\right. \;.
 $$

 Endow $X$ with the left-invariant almost-complex structure $J$ defined by the $(1,0)$-forms
 $$ \left\{\begin{array}{rcl}
            \varphi^1 &:=& e^1 + \im\, e^2 \\[5pt]
            \varphi^2 &:=& f^1 + \im\, f^2 \\[5pt]
            \varphi^3 &:=& e^3 + \im\, f^3
           \end{array}\right. \;.
 $$
 By computing the complex structure equations,
 $$ \left\{\begin{array}{rcl}
            \del \varphi^1 &=& \im\, \varphi^1 \wedge \varphi^3 \\[5pt]
            \del \varphi^2 &=& \varphi^2 \wedge \varphi^3 \\[5pt]
            \del \varphi^3 &=& 0
           \end{array}\right.
    \qquad \text{ and } \qquad
    \left\{\begin{array}{rcl}
            \delbar \varphi^1 &=& \im\, \varphi^1 \wedge \bar\varphi^3 \\[5pt]
            \delbar \varphi^2 &=& - \varphi^2 \wedge \bar\varphi^3 \\[5pt]
            \delbar \varphi^3 &=& - \im\, \varphi^1 \wedge \bar\varphi^1 + \varphi^2\wedge\bar\varphi^2
           \end{array}\right. \;,
 $$
 we note that $J$ is in fact integrable.

 The manifold $X$ is a compact complex manifold not admitting any K\"ahler metric. (Indeed, the second Betti number is zero.) The above complex structure coincides with the structures that have been studied by E. Calabi and B. Eckmann on the products $\mathbb{S}^{2p+1}\times\mathbb{S}^{2q+1}$ as clarifying examples in non-K\"ahler geometry, \cite{calabi-eckmann}.

 Consider the Hermitian metric $g$ whose associated $(1,1)$-form is
 $$ \omega \;:=\; \frac{\im}{2}\, \sum_{j=1}^{3} \varphi^j \wedge \bar\varphi^j \;. $$

 As for the de Rham cohomology, from the K\"unneth formula we get
 $$ H^\bullet_{dR}(X;\C) \;=\; \C \left\langle 1 \right\rangle \oplus \C\left\langle \varphi^{13\bar1}-\varphi^{1\bar1\bar3},\, \varphi^{23\bar2}+\varphi^{2\bar2\bar3} \right\rangle \oplus \C\left\langle \varphi^{123\bar1\bar2\bar3} \right\rangle \;, $$
 (where, here and hereafter, we shorten, e.g., $\varphi^{13\bar1} := \varphi^1\wedge\varphi^3\wedge\bar\varphi^1$).

 By \cite[Appendix II, Theorem 9.5]{hirzebruch}, one has that a model for the Dolbeault cohomology is given by
 $$ H_{\delbar}^{\bullet,\bullet}(X) \;\simeq\; \left. \C\left[y^{(1,1)}\right] \middle\slash \left(\left(y^{(1,1)}\right)^{2}\right) \right. \otimes \wedge^{\bullet,\bullet} \left( \C\left\langle z \right\rangle^{(2,1)} \oplus \C\left\langle x \right\rangle^{(0,1)} \right) \;, $$
 where superscripts denote bi-degree.
 In particular, we recover that the Hodge numbers are
 $$
 \begin{array}{ccccccc}
  &  &  & h^{0,0}_{\delbar}=1 &  &  &  \\
  &  & h^{1,0}_{\delbar}=0 &  & h^{0,1}_{\delbar}=1 &  &  \\
  & h^{2,0}_{\delbar}=0 &  & h^{1,1}_{\delbar}=1 &  & h^{0,2}_{\delbar}=0 &  \\
 h^{3,0}_{\delbar}=0 &  & h^{2,1}_{\delbar}=1 &  & h^{1,2}_{\delbar}=1 &  & h^{0,3}_{\delbar}=0 \\
  & h^{3,1}_{\delbar}=0 &  & h^{2,2}_{\delbar}=1 &  & h^{1,3}_{\delbar}=0 &  \\
  &  & h^{3,2}_{\delbar}=1 &  & h^{2,3}_{\delbar}=0 &  &  \\
  &  &  & h^{3,3}_{\delbar}=1 &  &  &  \\
 \end{array} \;.
 $$
 Consider the sub-complex
 $$ \iota \colon \wedge \left\langle \varphi^1,\, \varphi^2,\, \varphi^3,\, \bar\varphi^1,\, \bar\varphi^2,\, \bar\varphi^3 \right\rangle \hookrightarrow \wedge^{\bullet,\bullet}X \;. $$
 Since it is closed for the $\C$-linear Hodge-$*$-operator associated to $g$, $H_{\delbar}(\iota)$ is injective, see, e.g., \cite[Lemma 9]{console-fino}, compare also \cite[Theorem 1.6, Remark 1.9]{angella-kasuya-1}. By knowing the Hodge numbers, we note that is it also surjective. Therefore $H_{\delbar}(\iota)$ is an isomorphism. More precisely, we get
 \begin{eqnarray*}
  H^{\bullet,\bullet}_{\delbar}(X) &=& \C\left\langle 1 \right\rangle \oplus \C\left\langle \left[\varphi^{\bar{3}}\right] \right\rangle \oplus \C\left\langle \left[ \im\, \varphi^{1\bar1}+\varphi^{2\bar2} \right] \right\rangle \\[5pt]
   && \oplus \C\left\langle \left[ \varphi^{23\bar2}+\im\,\varphi^{13\bar1}\right] \right\rangle \oplus \C\left\langle \left[\im\varphi^{1\bar1\bar3}+\varphi^{2\bar2\bar3}\right] \right\rangle \\[5pt]
   && \oplus \C\left\langle \left[\im\,\varphi^{23\bar2\bar3}+\varphi^{13\bar1\bar3}\right] \right\rangle \oplus \C\left\langle \left[\varphi^{123\bar1\bar2}\right] \right\rangle \oplus \C\left\langle \left[\varphi^{123\bar1\bar2\bar3}\right] \right\rangle \;,
 \end{eqnarray*}
 where we have listed the $\Delta_{\delbar}^{g}$-harmonic representatives.
 
 Note that the Hermitian metric associated to $\omega$ is not geometrically-Dolbeault-formal: indeed, $\left( \im\, \varphi^{1\bar1}+\varphi^{2\bar2} \right) \wedge \left( \im\, \varphi^{1\bar1}+\varphi^{2\bar2} \right) = -2\,\im\, \varphi^{12\bar1\bar2}$ is not $\Delta_{\delbar}$-harmonic. On the other hand, $X$ is Dolbeault formal in the sense of Neisendorfer and Taylor as proven in \cite[page 188]{neisendorfer-taylor}.

 By \cite[Theorem 1.3, Proposition 2.2]{angella-kasuya-1}, we have also that $H_{BC}(\iota)$ is an isomorphism. In particular, we get
 \begin{eqnarray*}
  H^{\bullet,\bullet}_{BC}(X) &=& \C\left\langle 1 \right\rangle \oplus \C\left\langle \left[\varphi^{1\bar1}\right],\, \left[\varphi^{2\bar2}\right] \right\rangle \oplus \C\left\langle \left[\varphi^{23\bar2}+\im\,\varphi^{13\bar1}\right] \right\rangle \\[5pt]
   && \oplus \C\left\langle \left[\varphi^{2\bar2\bar3}-\im\,\varphi^{1\bar1\bar3}\right] \right\rangle \oplus \C\left\langle \left[ \varphi^{12\bar1\bar2}\right] \right\rangle \oplus \C\left\langle \left[\varphi^{123\bar1\bar2}\right] \right\rangle \\[5pt]
   && \oplus \C\left\langle \left[\varphi^{12\bar1\bar2\bar3}\right] \right\rangle \oplus \C\left\langle \left[\varphi^{123\bar1\bar2\bar3}\right] \right\rangle \;,
 \end{eqnarray*}
 where we list the harmonic representatives with respect to the Bott-Chern Laplacian of the Hermitian metric associated to $\omega$.

 In particular, the Bott-Chern numbers are
 $$
 \begin{array}{ccccccc}
  &  &  & h^{0,0}_{BC}=1 &  &  &  \\
  &  & h^{1,0}_{BC}=0 &  & h^{0,1}_{BC}=0 &  &  \\
  & h^{2,0}_{BC}=0 &  & h^{1,1}_{BC}=2 &  & h^{0,2}_{BC}=0 &  \\
 h^{3,0}_{BC}=0 &  & h^{2,1}_{BC}=1 &  & h^{1,2}_{BC}=1 &  & h^{0,3}_{BC}=0 \\
  & h^{3,1}_{BC}=0 &  & h^{2,2}_{BC}=1 &  & h^{1,3}_{BC}=0 &  \\
  &  & h^{3,2}_{BC}=1 &  & h^{2,3}_{BC}=1 &  &  \\
  &  &  & h^{3,3}_{BC}=1 &  &  &  \\
 \end{array} \;.
 $$

 Note that the product of Bott-Chern harmonic forms is still Bott-Chern harmonic, hence $\omega$ is a geometrically-Bott-Chern-formal metric on $X$.
\end{proof}

\subsection{Holomorphically-parallelizable Nakamura manifold}\label{subsec:hol-par-nakamura}
Consider the {\em holo\-mor\-phically-parallelizable Nakamura manifold}, \cite[\S2]{nakamura}, namely, $X:=\left. \Gamma \middle\backslash G \right.$ where
$$ G \;:=\; \C \ltimes_\phi \C^2 \qquad \text{ with } \qquad \phi(z) \;:=\; \left(\begin{array}{cc}\mathrm{e}^z&0\\0&\mathrm{e}^{-z}\end{array}\right) \;,$$
and $\Gamma$ is a lattice in $G$. In \cite[Example 2]{kasuya-mathz}, respectively \cite[Example 2.26]{angella-kasuya-1}, the Dolbeault, respectively Bott-Chern, cohomology of $X$ is computed, depending on the lattice, by using the techniques developed by H. Kasuya in \cite{kasuya-mathz}, and the results in \cite{angella-kasuya-1}. More precisely, \cite[Corollary 1.3, Corollary 4.2]{kasuya-mathz} provides a differential $\Z^2$-graded algebra
$$ \iota \colon \left( B^{\bullet,\bullet}_{\Gamma},\, \delbar \right) \hookrightarrow \left( \wedge^{\bullet,\bullet}X,\, \delbar \right) \;, $$
being finite-dimensional as a $\C$-vector space, such that $H_{\delbar}(\iota)$ is an isomorphism. In fact, as remarked to us by H. Kasuya, $\left( B^{\bullet,\bullet}_{\Gamma},\, \del,\, \delbar \right)$ has a structure of bi-differential $\Z^2$-graded algebra. As for Bott-Chern cohomology, by \cite[Corollary 2.15]{angella-kasuya-1}, the inclusion
$$ \iota \colon \left( C^{\bullet,\bullet}_{\Gamma},\, \del,\, \delbar \right) \hookrightarrow \left( \wedge^{\bullet,\bullet}X,\, \del,\, \delbar \right) \quad \text{ where } \quad C^{\bullet,\bullet}_{\Gamma} \;:=\; B^{\bullet,\bullet}_{\Gamma} + \overline{ B^{\bullet,\bullet}_{\Gamma} } $$
of bi-differential $\Z^2$-graded $\C$-vector spaces is such that $H_{BC}(\iota)$ is an isomorphism.

\medskip
 
We show that the natural Hermitian metric on the holomorphically-parallelizable Nakamura manifold is not geometrically-Bott-Chern-formal, but it is geometrically-Dolbeault-formal.

\begin{prop}
Let $X = \left. \Gamma \middle\backslash G \right.$ be the holomorphically-parallelizable Nakamura manifold. The Hermitian metric $g := \de z_1 \odot \de \bar z_1 + \esp^{-2z_1} \de z_2 \odot \de \bar z_2 + \esp^{2z_1} \de z_3 \odot \de \bar z_3$ is geometrically-Dolbeault-formal but non-geometrically-Bott-Chern-formal.
\end{prop}

\begin{proof}
 To describe explicitly the cohomologies of the holomorphically-parallelizable Nakamura manifold, depending on the lattice, consider a set $\left\{ z_1, z_2, z_3 \right\}$ of local holomorphic coordinates, where $\left\{ z_1 \right\}$ and $\left\{ z_2, z_3 \right\}$ give local holomorphic coordinates on $\C$ and, respectively, $\C^2$. (As a matter of notation, we shorten, e.g., $\esp^{-z_1}\de z_{1\bar2} := \esp^{-z_1}\de z_{1} \wedge \de\bar z_{2}$.)
 The lattice is of the form
 $$ \Gamma \;=\; \left( \Z \left( a+\sqrt{-1}\,b \right) + \Z \left( c+\sqrt{-1}\,d \right) \right) \ltimes_{\phi} \Gamma^{\prime\prime} $$
 where $\Gamma^{\prime\prime}$ is a lattice of $\C^{2}$, and $a+\sqrt{-1}\,b \in \C$ and $c+\sqrt{-1}\,d\in \C$ are such that $ \Z(a+\sqrt{-1}\,b)+\Z(c+\sqrt{-1}\,d)$ is a lattice in $\C$ and $\phi(a+\sqrt{-1}\,b)$ and $\phi(c+\sqrt{-1}\,d)$ are conjugate to elements of $\mathrm{SL}(4;\Z)$, where we regard $\mathrm{SL}(2;\C)\subset \mathrm{SL}(4;\R)$, see \cite{hasegawa-dga}.
 We distinguish the following two cases.

 \begin{description}
  \item[case {\itshape (a)}] Suppose that
  $$ b\in\pi\,\Z \quad \text{ and } \quad d \in \pi\,\Z \;. $$
  In this case, the Dolbeault cohomology of $X$ is computed by means of the bi-differential $\Z^2$-graded algebra
  \begin{eqnarray*}
  B^{\bullet,\bullet}_{\Gamma} &=& \wedge^{\bullet,\bullet} \left( \C\left\langle \de z_1,\; \esp^{- z_1}\de z_2,\; \esp^{z_1}\de z_3 \right\rangle \right. \\[5pt]
  && \left. \oplus \C\left\langle \de z_1,\; \esp^{- z_1}\de \bar z_2,\; \esp^{z_1}\de \bar z_3 \right\rangle \right) \;,
  \end{eqnarray*}
  see \cite[Example 2, page 446]{kasuya-mathz}.
  Note that
  $$ \left. \delbar \right\lfloor_{B^{\bullet,\bullet}_{\Gamma}} \;=\; 0 $$
  and that $B^{\bullet,\bullet}_{\Gamma}$ is closed for the Hodge-$*$-operator associated to the metric $g := \de z_1 \odot \de \bar z_1 + \esp^{-2 z_1} \de z_2 \odot \de \bar z_2 + \esp^{2z_1} \de z_3 \odot \de\bar z_3$. Therefore we have
  $$ H^{\bullet,\bullet}_{\delbar}(X) \;\simeq\; \ker \Delta_{\delbar}^{g} \;=\; B^{\bullet,\bullet}_{\Gamma} \;, $$
  see \cite[Example 2, page 446]{kasuya-mathz}.
  In particular, $X$ is geometrically-Dolbeault-formal. 

  As for the Bott-Chern cohomology, consider the $\Z^2$-graded vector space $\left( C^{\bullet,\bullet},\, \del,\, \delbar \right)$ as in \cite[Table 7]{angella-kasuya-1}.
  (Note that $C^{\bullet,\bullet}$ has in fact a structure of algebra.) Since $C^{\bullet,\bullet}$ is closed for the Hodge-$*$-operator associated to the metric $g$, it allows to compute the $\Delta_{BC}^{g}$-harmonic representatives of the Bott-Chern cohomology as done in \cite[Table 8]{angella-kasuya-1}.
  Note that, e.g.,
  $$ \esp^{-z_1}\,\de z_{12} \qquad \text{ and } \qquad \esp^{\bar z_1}\,\de z_{3\bar1} $$
  are $\Delta_{BC}^{g}$-harmonic forms but their product
  $$ \esp^{-z_1}\,\de z_{12} \wedge \esp^{\bar z_1}\,\de z_{3\bar1} \;=\; \esp^{-2\im\Im z_1}\,\de z_{123\bar1} $$
  is not. In particular, the metric $g$ on $X$ is not geometrically-Bott-Chern-formal.

  \item[case {\itshape (b)}] Suppose that
  $$ b\not \in \pi\,\Z \quad \text{ or } \quad d \not \in\pi\,\Z \;. $$
  In this case, the Dolbeault cohomology of $X$ is computed by means of the bi-differential $\Z^2$-graded algebra
  \begin{eqnarray*}
  B^{\bullet,\bullet}_{\Gamma} &=& \wedge^{\bullet,\bullet} \left( \C\left\langle \de z_1,\; \esp^{- z_1}\de z_2,\; \esp^{z_1}\de z_3 \right\rangle \right. \\[5pt]
  && \left. \oplus \C\left\langle \de \bar z_1 \right\rangle \oplus \C\left\langle \de \bar z_2 \wedge \de \bar z_3 \right\rangle \right) \;,
  \end{eqnarray*}
  see \cite[Example 2, page 447]{kasuya-mathz}.
  Note that
  $$ \left. \delbar \right\lfloor_{B^{\bullet,\bullet}_{\Gamma}} \;=\; 0 $$
  and that $B^{\bullet,\bullet}_{\Gamma}$ is closed for the Hodge-$*$-operator associated to the metric $g := \de z_1 \odot \de \bar z_1 + \esp^{-2 z_1} \de z_2 \odot \de \bar z_2 + \esp^{2z_1} \de z_3 \odot \de\bar z_3$. Therefore we have
  $$ H^{\bullet,\bullet}_{\delbar}(X) \;\simeq\; \ker \Delta_{\delbar}^{g} \;=\; B^{\bullet,\bullet}_{\Gamma} \;, $$
  see \cite[Example 2, page 447]{kasuya-mathz}.
  In particular, $X$ is geometrically-Dolbeault-formal.
  
  As for the Bott-Chern cohomology, consider the $\Z^2$-graded vector space $\left( C^{\bullet,\bullet},\, \del,\, \delbar \right)$ as in \cite[Table 10]{angella-kasuya-1}.
  (Note that $C^{\bullet,\bullet}$ does not have a structure of algebra: for example, $\esp^{-z_1}\,\de z_2 \in C^{1,0}$ and $\esp^{-\bar z_1}\,\de \bar z_2 \in C^{0,1}$ but $\esp^{-z_1}\,\de z_2 \wedge \esp^{-\bar z_1}\,\de \bar z_2 = \esp^{-2\Re z_1}\, \de z_{2\bar2} \not\in C^{1,1}$.) Since $C^{\bullet,\bullet}$ is closed for the Hodge-$*$-operator associated to the metric $g$, it allows to compute the $\Delta_{BC}^{g}$-harmonic representatives of the Bott-Chern cohomology as done in \cite[Table 11]{angella-kasuya-1}.
  Note that, e.g.,
  $$ \esp^{-z_1}\,\de z_{12} \qquad \text{ and } \qquad \esp^{-\bar z_1}\,\de z_{\bar1\bar2} $$
  are $\Delta_{BC}^{g}$-harmonic forms but their product
  $$ \esp^{-z_1}\,\de z_{12} \wedge \esp^{-\bar z_1}\,\de z_{\bar1\bar2} \;=\; \esp^{-2\Re z_1}\,\de z_{12\bar1\bar2} $$
  is not. In particular, the metric $g$ on $X$ is not geometrically-Bott-Chern-formal.
 \end{description}
This proves the statement.
\end{proof}

As regards the existence of non-zero triple Aeppli-Bott-Chern Massey products, one can argue as in \cite{tardini-tomassini}, showing that it depends on the choice of the lattice $\Gamma$.

\end{document}